\documentclass[11pt]{amsart}
\usepackage[foot]{amsaddr}
\usepackage{amsmath,amssymb,bm, graphicx,float}

\begin{document}
\newtheorem{theorem}{Theorem}[section]
\newtheorem{lemma}[theorem]{Lemma}
\newtheorem{corollary}[theorem]{Corollary}
\newtheorem{prop}[theorem]{Proposition}
\newtheorem{definition}[theorem]{Definition}
\newtheorem{remark}[theorem]{Remark}

 \def\ad#1{\begin{aligned}#1\end{aligned}}  \def\b#1{{\bf #1}} \def\hb#1{\hat{\bf #1}}
\def\a#1{\begin{align*}#1\end{align*}} \def\an#1{\begin{align}#1\end{align}}
\def\e#1{\begin{equation}#1\end{equation}} \def\tx#1{\hbox{\rm{#1}}}\def\t#1{\operatorname{#1}}
\def\dt#1{\left|\begin{matrix}#1\end{matrix}\right|}
\def\p#1{\begin{pmatrix}#1\end{pmatrix}} \def\c{\operatorname{curl}}
 \def\vc{\nabla\times } \numberwithin{equation}{section}
 \def\la{\circle*{0.25}}
\def\boxit#1{\vbox{\hrule height1pt \hbox{\vrule width1pt\kern1pt
     #1\kern1pt\vrule width1pt}\hrule height1pt }}
 \def\lab#1{\boxit{\small #1}\label{#1}}
  \def\mref#1{\boxit{\small #1}\ref{#1}}
 \def\meqref#1{\boxit{\small #1}\eqref{#1}}
\long\def\comment#1{}

\def\lab#1{\label{#1}} \def\mref#1{\ref{#1}} \def\meqref#1{\eqref{#1}}

\def\bg#1{{\pmb #1}} 

\title  [$C^1$ serendipity elements]
   {$C^1$-$Q_k$ serendipity finite elements on rectangular meshes}

\author { Shangyou Zhang }
\address{Department of Mathematical  Sciences, University of Delaware, Newark, DE 19716, USA.}
\email{ szhang@udel.edu }

\date{}

\begin{abstract}
A $C^1$-$Q_k$ serendipity finite element is a sub-element of $C^1$-$Q_k$ 
   BFS finite element such that the element remains $C^1$-continuous and includes all
     $P_k$ polynomials.
In other words, it is a minimum of $Q_k$ bubbles enriched $P_k$ finite element.
We enrich the $P_4$ and $P_5$ spaces by $9$ $Q_4$ and $11$ $Q_5$-bubble functions, respectively.
For all $k\ge 6$, we enrich the $P_k$ spaces exactly by $12$ $Q_k$ bubble functions.
We show the uni-solvence and quasi-optimality of the newly defined $C^1$-$Q_k$ serendipity elements. 
Numerical experiments by the $C^1$-$Q_k$ serendipity elements, $4\le k\le 8$, are performed.
  
\end{abstract}

\vskip .3cm



\maketitle

\section{Introduction} 
The finite element methods became popular after some engineers and mathematicians
   started the constructions for the following biharmonic equation, ie. 
   the plate bending equation,
\an{\label{bi} \ad{ \Delta^2 u & = f \quad \t{in } \ \Omega, \\
        u=\partial_{\b n} u & =0  \quad \t{on } \ \partial\Omega, } }
where $\Omega$ is a polygonal domain in 2D, and $\b n$ is a normal vector.
We mention some important constructions in the early days,  
 the $C^1$-$P_3$ Hsieh-Clough-Tocher element (1961,1965) \cite{Ciarlet,Clough}, 
  the $C^1$-$P_3$ Fraeijs de Veubeke-Sander element (1964,1965) \cite{Fraeijs,Fraeijs68,Sander}  
 the $C^1$-$P_5$  Argyris element (1968) \cite{Argyris},
  the $C^1$-$P_4$ Bell element (1969) \cite{Bell},
   the $C^1$-$Q_3$ Bogner-Fox-Schmit element (1965) \cite{Bogner},  and
 the $P_2$ nonconforming Morley  element (1969) \cite{morley}.
 
The $C^1$-$P_3$ Hsieh-Clough-Tocher element was extended to the $C^1$-$P_k$ ($k\ge 3$) 
   finite elements in \cite{Douglas,ZhangMG}.
The $C^1$-$P_5$  Argyris element was extended to the family of 
   $C^1$-$P_k$ ($k\ge 5$) finite elements in \cite{Zen70,Zlamal}. 
The $C^1$-$P_5$  Argyris element was modified and extended to the family of  
   $C^1$-$P_k$ ($k\ge 5$) full-space finite elements in \cite{Morgan-Scott}.
The $C^1$-$P_5$  Argyris element was also extended to 3D $C^1$-$P_k$ ($k\ge 9$)
   elements on tetrahedral meshes in \cite{Zenisek,Z3d,Z4d}.
The $C^1$-$P_4$ Bell element was extended to three families of 
    $C^1$-$P_{2m+1}$ ($m\ge 3$) finite elements in \cite{Xu-Zhang7,Xu-Zhang}. 
The Bell finite elements do not have any degrees of freedom on edges. 
Thus they must be odd-degree polynomials 
  (the $P_4$ Bell element is a subspace of $P_5$ polynomials.)
The $C^1$-$Q_3$ Bogner-Fox-Schmit element was extended to three families of $C^1$-$Q_k$ ($k\ge 3$)
  finite elements on rectangular meshes in \cite{Zhang-C1Q}.
The $C^1$-$P_3$ Fraeijs de Veubeke-Sander element is extended to two families
  of $C^1$-$P_k$ ($k\ge 3$) finite elements in \cite{Zhang-F}. 

In this work, we extend the $C^1$-$Q_3$ Bogner-Fox-Schmit element to $C^1$-$P_k$ ($k\ge 3$) serendipity
  finite elements.
That is, we enrich the $P_k$ polynomial by a minimum number of $Q_k$ bubble functions to construct
   $C^1$ finite elements on rectangular meshes.

On 2D rectangular meshes, the $C^0$-$P_k$ serendipity finite element is defined by a two-$Q_k$-bubble
   enrichment on each rectangle $T$:
\a{ S_k(T) = P_k(T) + \t{span} \{ x^k y, x y^k \}, \quad k\ge 1. }
cf. \cite{Arnold}.  For the lowest degree case $k=1$, $S_1(T)$ is $Q_1(T)$, 
   the set of bilinear polynomials.
The construction of 3D rectangular serendipity finite elements is completed by
    Arnold and Awanou, in \cite{Arnold}.
     
For the $C^1$-$Q_3$ BFS finite element, all degrees of freedom are on the boundary of a rectangle.
Thus, the $C^1$-$Q_3$ serendipity finite element is the $C^1$-$Q_3$ BFS finite element itself.

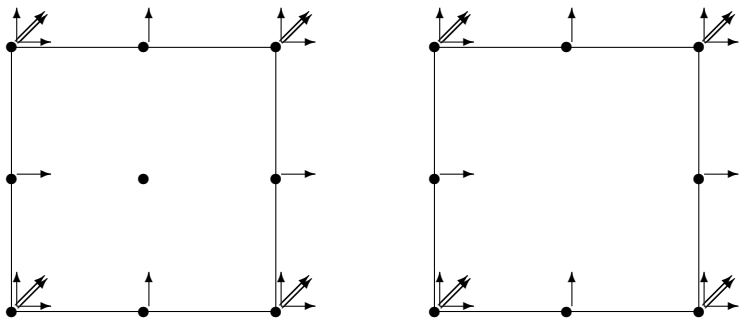
\begin{figure}[H]\centering
 \begin{picture}(300,140)(0,-10) 
 \def\c{\circle*{4}}\def\h{\vector(1,0){13}}\def\v{\vector(0,1){13}}
 \def\d{\multiput(0.5,-0.5)(-1,1){2}{\vector(1,1){11}}}  

 \put(0,0){\begin{picture}(100,108)(0,0)   
  \put(0,0){\line(1,0){100}}  \put(0,0){\line(0,1){100}}   
  \put(100,100){\line(-1,0){100}}  \put(100,100){\line(0,-1){100}}  
  \multiput(0,0)(50,0){3}{\multiput(0,0)(0,50){3}{\c}}  
  \multiput(2,2)(50,0){3}{\v}  \multiput(2,102)(50,0){3}{\v}  
  \multiput(2,2)(0,50){3}{\h}  \multiput(102,2)(0,50){3}{\h} 
    \multiput(2,2)(0,100){2}{\d} \multiput(102,2)(0,100){2}{\d} 
  
 \end{picture} }
  
 \put(160,0){\begin{picture}(100,108)(0,0)   
  \put(0,0){\line(1,0){100}}  \put(0,0){\line(0,1){100}}   
  \put(100,100){\line(-1,0){100}}  \put(100,100){\line(0,-1){100}}  
  \multiput(0,0)(100,0){2}{\multiput(0,0)(0,50){3}{\c}} 
     \multiput(50,0)(100,0){1}{\multiput(0,0)(0,100){2}{\c}}  
  \multiput(2,2)(50,0){3}{\v}  \multiput(2,102)(50,0){3}{\v}  
  \multiput(2,2)(0,50){3}{\h}  \multiput(102,2)(0,50){3}{\h} 
    \multiput(2,2)(0,100){2}{\d} \multiput(102,2)(0,100){2}{\d} 
   
 \end{picture} }
  
 \end{picture}
 \caption{Left: The 25 degrees of freedom 
     for the $C^1$-$Q_4$ BFS element in \eqref{d-Q-k}; \ 
     Right: The 24 degrees of freedom for the $C^1$-$P_4$
       serendipity finite element in \eqref{d-4}. } \label{T-3}
 \end{figure}

To define the $C^1$-$Q_4$ (also referred as $C^1$-$P_4$) serendipity finite element,
  we eliminate the only one internal degree of freedom from the set of
  25 degrees of freedom of the $C^1$-$Q_4$ BFS finite element, shown in Figure \ref{T-3}.
 Though reducing only 1/25 unknowns locally, we have about a 1/10 reduction in the number
   of global unknowns.

Next, to define the $C^1$-$Q_5$ serendipity finite element, 
  we remove all 4 internal degrees of freedom in the set of 36 dofs of the $C^1$-$Q_5$ element.
The local and global ratios of the reduction are about 1/9 and 1/4, respectively.

For the $C^1$-$Q_6$ and $Q_7$ serendipity elements, we eliminate internal $3^2=9$ and $4^2=16$
  dofs from the original $7^2=49$ and $8^2=64$ dofs, respectively.
The global reduction is close the maximal rate of one half.

For $k\ge 8$,  we cannot remove all internal $(k-3)^2$ degrees of freedom in
  the $C^1$-$Q_k$ finite element.
 This is understandable as 8 lines of information 
     (from $C^1$ dofs on the 4 edges of a rectangle) 
   is not enough to determine a $P_8$ polynomial.
Thus we keep the internal $P_{k-8}(T)$ Lagrange nodes of dofs, for $C^1$-$Q_k$ ($k\ge 8$)
    serendipity elements.

As discussed above,  in this work,  we construct a family of $C^1$-$Q_k$ ($k\ge 4$)
    serendipity elements.
To ensure (1) $C^1$-continuity, (2) $P_k$-inclusion and (3) $Q_k$-subset,
  we enrich the $P_4$ and $P_5$ spaces by $9$ $Q_4$ and $11$ $Q_5$-bubble functions, respectively.
For all $k\ge 6$, we enrich the $P_k$ spaces exactly by $12$ $Q_k$ bubble functions.
We show the uni-solvence and quasi-optimality of the newly defined $C^1$-$Q_k$ serendipity elements.   
Numerical tests on the new $C^1$-$P_k$, $k=4,5,6,7$ and $8$, serendipity elements are performed 
  and their comparisons with the corresponding $C^1$-$Q_k$ elements are provided,
  confirming the theory.

\section{The $C^1$-$P_4$ serendipity finite element}

Let $\mathcal Q_h=\{ T \}$ be a uniform square mesh on the domain $\Omega$.
The standard $C^1$-$Q_k$ Bogner-Fox-Schmit (BFS) finite element space on $\mathcal Q_h$ is
  defined by
\an{\label{W-h} W_h = \{ u_h\in  H^2_0(\Omega) : u_h|_T \in Q_k(T) \ \forall T\in \mathcal T_h \}, }
where $Q_k(T)$ is the set of polynomials of separated degree $k$ or less.
 
We define the degrees of freedom of the  $C^1$-$Q_k$ BFS element, $k\ge 3$, cf. Figure \ref{T-3},
     by $F_m(p)=$
\an{\label{d-Q-k} &\begin{cases} p(\b x_i), \partial_x p(\b x_i), \partial_y p(\b x_i),
        \partial_{xy} p(\b x_i),  & i=1,2,3,4, \\
    p(\frac { j\b x_1+j'\b x_{2}}{k-2}),\;
    \partial_{y} p(\frac { j\b x_1+j'\b x_{2}}{k-2}), & j=1,\dots, k-3, \\ 
    p(\frac { j\b x_2+j'\b x_{3}}{k-2}), \;
    \partial_{x} p(\frac { j\b x_2+j'\b x_{3}}{k-2}), & j=1,\dots, k-3, \\  
     p(\frac { j\b x_4+j'\b x_{3}}{k-2}), \;
    \partial_{y} p(\frac { j\b x_4+j'\b x_{3}}{k-2}), & j=1,\dots, k-3, \\ 
    p(\frac { j\b x_1+j'\b x_{4}}{k-2}), \;
    \partial_{x} p(\frac { j\b x_1+j'\b x_{4}}{k-2}), & j=1,\dots, k-3, \\ 
     p(\frac { ( j\b x_1+j'\b x_{4})\ell+( j\b x_2+j'\b x_{3})\ell'}{(k-2)^2}), 
            & j,\ell =1,\dots, k-3, 
    \end{cases} }
where $j'=k-2-j$, $\ell'=k-2-\ell$, and $\b x_i$ are the four vertices of $T$ as shown in 
   Figure \ref{T}.

\begin{lemma} The degrees of freedom \eqref{d-Q-k} uniquely determine the 
   $Q_k(T)$ functions in \eqref{W-h}.
\end{lemma}

\begin{proof}We count the dimension of $Q_k(T)$ and the number $N_{\text{dof}}$ of
   degrees of freedom in \eqref{d-Q-k},
\a{ \dim Q_k(T) &= (k+1)^2 =k^2+2k+1, \\
    N_{\text{dof}} &= 16+8(k-3)+(k-3)^2 =k^2+2k+1. }
Thus the uni-solvency is determined by uniqueness.

Let $p_k\in Q_k(T)$ and $F_m(p_k)=0$ for all degrees of freedom in \eqref{d-Q-k}.
Evaluating the $(k+1)$ degrees of freedom,
   the function values and the two $\partial_x$ derivatives at the two end points on $\b x_1\b x_2$,
  we get $p_k|_{\b x_1\b x_2}=0$ and
\a{ p_k=\frac{y-y_1}h p_{k,k-1}\quad\t{for some } \ p_{k,k-1}\in Q_{k,k-1}(T),  }
where $h=y_4-y_1$, $(x_1,y_1)=\b x_1$ and $Q_{k,k-1}$ is the space of separated degrees $k$ and $k-1$
  in $x$ and $y$ respectively.
By the $(k-1)$ $\partial_y p_k$ and 2 $\partial_{xy} p_k$ dofs at $\b x_1\b x_2$, we get 
  $p_{k,k-1}|_{\b x_1\b x_2}=0$ and
\a{ p_k=\frac{(y-y_1)^2}{h^2} p_{k,k-2}\quad\t{for some } \ p_{k,k-2}\in Q_{k,k-2}(T).  }
 Repeating the argument on $\b x_4\b x_3$,  we get
\a{ p_k=\frac{(y-y_1)^2}{h^2}\frac{(y_4-y)^2}{h^2} p_{k,k-4}\quad\t{for some } \ p_{k,k-4}\in Q_{k,k-4}(T).  }
If $k=3$, the proof is done as $p_k=0$.

Evaluating the degrees of freedom at the line $y=y_{14,1}:= (y_1+(k-3)y_4)/(k-2)$,  we get
\a{ p_k(\frac { j\b x_1+j'\b x_{2}}{k-2})&=\frac{1}{(k-2)^2}\cdot \frac{(k-3)^2}{(k-2)^2}\cdot
              p_k(\frac { j\b x_1+j'\b x_{2}}{k-2})\\ &=0, \qquad\qquad j=0, \dots, k-2, \\ \\
    \partial_x p_k(\frac { j\b x_1+j'\b x_{2}}{k-2})&=\frac{1}{(k-2)^2}\cdot \frac{(k-3)^2}{(k-2)^2}\cdot
              \partial_x p_k(\frac { j\b x_1+j'\b x_{2}}{k-2})\\ & =0, \qquad\qquad j=0,  k-2, }
and $p_{k,k-4}|_{y=y_{14,1}}=0$.  Thus, 
we have
\a{ p_k=\frac{(y-y_1)^2}{h^2}\frac{(y_4-y)^2}{h^2}(y-y_{14,1}) p_{k,k-5} }
for some $p_{k,k-5}\in Q_{k,k-4}(T)$.   Repeating the evaluation on each line, we get 
\a{ p_k=\frac{(y-y_1)^2}{h^2}\frac{(y_4-y)^2}{h^2}\prod_{j=1}^{k-3}(y-y_{14,j}) p_{k,-1} }
for some $p_{k,-1}\in Q_{k,-1}(T)$.  Thus, $p_k=0$ and the lemma is proved.             
\end{proof}

Let $\{ b_i \}$ be the dual basis of $W_h$ on $T$, to the degrees of freedom in \eqref{d-Q-k}.
For $k=4$, we select 9 bubble basis functions 
  $\{b_5, b_6, b_7, b_8, b_{12}, b_{14}, b_{17}, b_{18}$, $b_{20}\}$
    as shown in Figure \ref{T}.
Enriched by the nine bubble functions,  we define the $C^1$-$P_4$ serendipity element by
\an{\label{V-4} V_4(T)=\t{span} \{ P_4(T), \ b_j, \ j=5, 6, 7, 8, {12}, {14}, {17}, {18}, 20\}.  }
We define the following degrees of freedom for the space $V_4(T)$, ensuring the global 
   $C^1$ continuity late, \  by $F_m(p)=$
\an{\label{d-4} &\begin{cases} p(\b x_i), \partial_x p(\b x_i), \partial_y p(\b x_i),
        \partial_{x y} p(\b x_i), \quad \ i=1,2,3,4, \\
    p(\frac { \b x_1+\b x_{2}}{2}),\;
    \partial_{y} p(\frac { \b x_1+\b x_{2}}{2}), \ 
    p(\frac { \b x_2+\b x_{3}}{2}), \;
    \partial_{x} p(\frac { \b x_2+\b x_{3}}{2}), &  \\  
     p(\frac { \b x_4+\b x_{3}}{2}), \;
    \partial_{y} p(\frac { \b x_4+\b x_{3}}{2}), \
    p(\frac { \b x_1+\b x_{4}}{2}), \;
    \partial_{x} p(\frac {\b x_1+\b x_{4}}{2}).
    \end{cases} }

\begin{lemma} The degrees of freedom \eqref{d-4} uniquely determine the 
   $V_4(T)$ functions in \eqref{V-4}.
\end{lemma}

\begin{proof}We count the dimension of $V_4$ in \eqref{V-4} and the number $N_{\text{dof}}$ of
   degrees of freedom in \eqref{d-4},
\a{ \dim V_4(T) &= \dim P_4 + 11 =15+9=24, \\
    N_{\text{dof}} &= 16+8 =24. }
Thus the uni-solvency is determined by uniqueness.

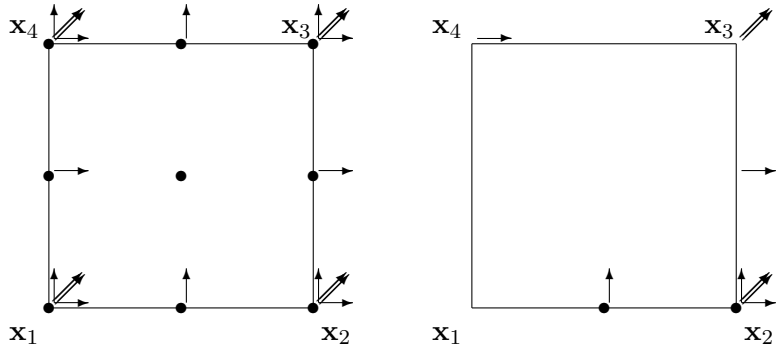
\begin{figure}[H]\centering
 \begin{picture}(280,140)(0,-10) 
 \def\c{\circle*{4}}\def\h{\vector(1,0){13}}\def\v{\vector(0,1){13}}
 \def\d{\multiput(0.5,-0.5)(-1,1){2}{\vector(1,1){11}}}  

 \put(0,0){\begin{picture}(100,108)(0,0)   
  \put(0,0){\line(1,0){100}}  \put(0,0){\line(0,1){100}}   
  \put(100,100){\line(-1,0){100}}  \put(100,100){\line(0,-1){100}}  
  \multiput(0,0)(50,0){3}{\multiput(0,0)(0,50){3}{\c}}  
  \multiput(2,2)(50,0){3}{\v}  \multiput(2,102)(50,0){3}{\v}  
  \multiput(2,2)(0,50){3}{\h}  \multiput(102,2)(0,50){3}{\h} 
    \multiput(2,2)(0,100){2}{\d} \multiput(102,2)(0,100){2}{\d} 
  
      \put(-15,-12){$\b x_1$} \put(103,-12){$\b x_2$}
      \put(-15,104){$\b x_4$} \put(88,104){$\b x_3$}
 \end{picture} }
 
 \put(160,0){\begin{picture}(100,108)(0,0)   
  \put(0,0){\line(1,0){100}}  \put(0,0){\line(0,1){100}}   
  \put(100,100){\line(-1,0){100}}  \put(100,100){\line(0,-1){100}}  \put(102,2){\v} 
   \put(102,102){\d}  \put(102,2){\h} \put(102,2){\d} 
  \put(2,102){\h}  \put(52,2){\v}  \put(50,0){\c} \put(100,0){\c} \put(102,52){\h} 
      \put(-15,-12){$\b x_1$} \put(103,-12){$\b x_2$}
      \put(-15,104){$\b x_4$} \put(88,104){$\b x_3$}
 \end{picture} }

 \end{picture}
 \caption{The 25 degrees of freedom 
     for the $C^1$-$Q_4$ BFS element in \eqref{d-Q-k}, 
     and the 9 bubble functions $\{b_5, b_6, b_7, b_8, b_{12}, b_{14}, b_{17}, b_{18}, b_{20}\}$
      used to define $C^1$-$P_4$ serendipity element in \eqref{V-4}. } \label{T}
 \end{figure}

Let $p\in V_4(T)$ in \eqref{V-4} and $F_m(p)=0$ for all degrees of freedom in \eqref{d-4}.
Let \an{\label{p-4} p=p_4+\sum_{j=1}^9 c_j b_{i_j} \quad \ \t{for some } \ p_4\in P_4(T).  }
As all $b_i$ vanish at these points, we have
\an{\label{p-4-1}\ad{
   && p_4(\b x_1)&=0, \ & \partial_y p_4(\b x_1)&=0, \ & p_4(\frac{\b x_1+\b x_4}2)&=0, & \\
    &&  p_4(\b x_4)& =0, \ & \partial_y  p_4(\b x_4)&=0,  } }
and consequently $p_4|_{\b x_1\b x_4}=0$ as the degree 4 polynomial has 5 zero points.
Thus \a{ p_4=\lambda_{14} p_3 \quad \ \t{for some } \ p_3\in P_3(T), }
where $\lambda_{14}$ is a linear polynomial vanishing at the line $\b x_1\b x_4$
  and assuming value $1$ at $\b x_2$.
  
Now, as all $b_i$ have these vanishing degrees of freedom,  we have
\a{ \partial_x p_4(\b x_1)&= h p_3(\b x_1)=0, \\
    \partial_{x y} p_4(\b x_1)&= h \partial_y p_3(\b x_1)=0, \\
    \partial_x p_4(\frac{\b x_4+\b x_1}2)&= h  p_3(\frac{\b x_4+\b x_1}2)=0,  \\ 
    \partial_{x y} p_4(\b x_4)&= h \partial_y p_3(\b x_4)=0,  } 
     and consequently $p_3|_{\b x_1\b x_4}=0$.
We can then factor out another linear polynomial that
\an{\label{p-4-2} p_4= \lambda_{14}^2 p_2 \quad \ \t{for some } \ p_2\in P_2(T). }
As $b_i$ have these three degrees of freedom vanished,  
  we then have 
  \a{ p_4(\frac{\b x_4+\b x_3}2)&= \frac 1{2^2} \cdot p_2(\frac{\b x_4+\b x_3}2)=0, \\
    p_4( \b x_3) &=1 \cdot p_2( \b x_3)=0, \\
    \partial_x p_4( \b x_3) &= \frac 1{h^2} \cdot p_2( \b x_3)+1 \cdot \partial_x  p_2( \b x_3)=0, }
and consequently $p_2|_{\b x_4\b x_3}=0$.
We factor out this linear polynomial factor as
\a{ p_4= \lambda_{14}^2 \lambda_{43} p_1  \quad \ \t{for some } \ p_1\in P_1(T), }
where $\lambda_{43}$ is a linear polynomial vanishing at the line $\b x_4\b x_3$
  and assuming value $1$ at $\b x_1$.

As $b_i$ again have the following two degrees of freedom vanished,  
  we then have
\a{ p_1(\frac{\b x_4+\b x_3}2)=0, \ \partial_x p_2( \b x_3)=0,   }
and consequently $p_1|_{\b x_3\b x_4}=0$.
We factor out this last linear polynomial factor as
\a{ p_4= \lambda_{14}^2 \lambda_{43}^2 c \quad \ \t{for some } \ c\in P_0(T), }
where $\lambda_{43}$ is a linear polynomial vanishing at the line $\b x_4\b x_3$
  and assuming value $1$ at $\b x_1$.
Evaluating the last degree of freedom $\partial_y p(\frac{\b x_4+\b x_3}2)=0$, 
  we have
\a{ \partial_y p_4(\frac{\b x_4+\b x_3}2)=\frac 1{2^2}\cdot \frac 1 2 \cdot \frac {-1} h c =0, }
where $h$ is the size of square $T$.  Thus $c=0$ and $p_4=0$ in \eqref{p-4}.

As $p_4=0$, evaluating $p$ in \eqref{p-4} sequentially at the degrees of freedom of $b_{i_j}$,
  it follows that
  \a{ c_1=\dots=c_9 = 0.  }
The lemma is proved as $p=0$ in \eqref{p-4}.
\end{proof}

\section{The $C^1$-$P_5$ serendipity finite element}

Enriched by the eleven bubble functions,  we define the $C^1$-$P_5$ serendipity element by
\an{\label{V-5} V_5(T)=\t{span} \{ P_5(T), \ b_j, \ j=5, 6, 7, 8, {12}, {14}, 18, 19,20,21,22\},  }
where $b_{i_j}$ is a basis function in \eqref{W-h}, dual to the degrees of freedom in \eqref{d-Q-k}.
We define the following degrees of freedom for the space $V_5(T)$, ensuring the global 
   $C^1$ continuity late, \  by $F_m(p)=$
\an{\label{d-5} &\begin{cases} p(\b x_i), \partial_x p(\b x_i), \partial_y p(\b x_i),
        \partial_{x y} p(\b x_i),   \ & i=1,2,3,4, \\
    p(\frac { j\b x_1+j'\b x_{2}}{k-2}),\;
    \partial_{y} p(\frac { j\b x_1+j'\b x_{2}}{k-2}), & j=1,\dots, k-3, \\ 
    p(\frac { j\b x_2+j'\b x_{3}}{k-2}), \;
    \partial_{x} p(\frac { j\b x_2+j'\b x_{3}}{k-2}), & j=1,\dots, k-3, \\  
     p(\frac { j\b x_4+j'\b x_{3}}{k-2}), \;
    \partial_{y} p(\frac { j\b x_4+j'\b x_{3}}{k-2}), & j=1,\dots, k-3, \\ 
    p(\frac { j\b x_1+j'\b x_{4}}{k-2}), \;
    \partial_{x} p(\frac { j\b x_1+j'\b x_{4}}{k-2}), & j=1,\dots, k-3,
    \end{cases} }
where $k=5$, and $j'=2-j'$.

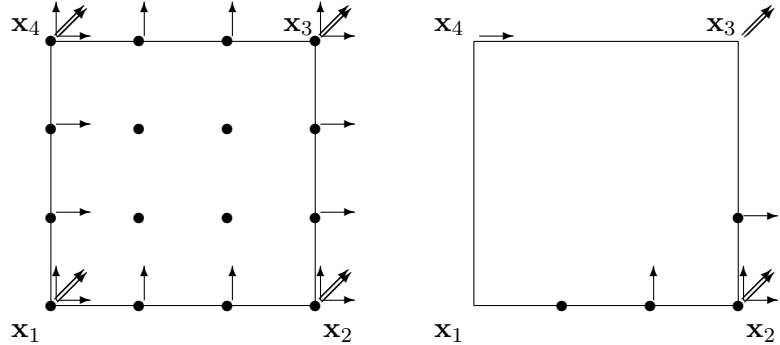
\begin{figure}[H]\centering
 \begin{picture}(280,140)(0,-10) 
 \def\c{\circle*{4}}\def\h{\vector(1,0){13}}\def\v{\vector(0,1){13}}
 \def\d{\multiput(0.5,-0.5)(-1,1){2}{\vector(1,1){11}}}  

 \put(0,0){\begin{picture}(100,108)(0,0)   
  \put(0,0){\line(1,0){100}}  \put(0,0){\line(0,1){100}}   
  \put(100,100){\line(-1,0){100}}  \put(100,100){\line(0,-1){100}}  
  \multiput(0,0)(33.33,0){4}{\multiput(0,0)(0,33.33){4}{\c}}  
  \multiput(2,2)(33.33,0){4}{\v}  \multiput(2,102)(33.33,0){4}{\v}  
  \multiput(2,2)(0,33.33){4}{\h}  \multiput(102,2)(0,33.33){4}{\h} 
    \multiput(2,2)(0,100){2}{\d} \multiput(102,2)(0,100){2}{\d} 
  
      \put(-15,-12){$\b x_1$} \put(103,-12){$\b x_2$}
      \put(-15,104){$\b x_4$} \put(88,104){$\b x_3$}
 \end{picture} }
 
 \put(160,0){\begin{picture}(100,108)(0,0)   
  \put(0,0){\line(1,0){100}}  \put(0,0){\line(0,1){100}}   
  \put(100,100){\line(-1,0){100}}  \put(100,100){\line(0,-1){100}}  \put(102,2){\v} 
   \put(102,102){\d}  \put(102,2){\h} \put(102,2){\d} 
  \put(2,102){\h}  \put(68,2){\v}  \put(33.33,0){\c} \put(66.66,0){\c} \put(100,0){\c}
  
     \put(102,34){\h} \put(100,33.33){\c} 
      \put(-15,-12){$\b x_1$} \put(103,-12){$\b x_2$}
      \put(-15,104){$\b x_4$} \put(88,104){$\b x_3$}
 \end{picture} }

 \end{picture}
 \caption{The $6\times 6$ degrees of freedom 
     for the $C^1$-$Q_5$ BFS element in \eqref{d-Q-k},
     and the 11 bubble functions $\{b_5, b_6, b_7, b_8, b_{12}, b_{14}, b_{18}, b_{19}, b_{20},b_{21},
        b_{22} \}$
      used to define $C^1$-$P_5$ serendipity element in \eqref{V-5}. } \label{T5}
 \end{figure}

\begin{lemma} The degrees of freedom \eqref{d-5} uniquely determine the 
   $V_5(T)$ functions in \eqref{V-5}.
\end{lemma}

\begin{proof} We count the dimension of $V_5$ in \eqref{V-5} and the number $N_{\text{dof}}$ of
   degrees of freedom in \eqref{d-5},
\a{ \dim V_5(T) &= \dim P_5 + 11 =21+11=32, \\
    N_{\text{dof}} &= 16+8\cdot 2=32. }
Thus the uni-solvency is determined by uniqueness.

Let $p\in V_5(T)$ in \eqref{V-5} and $F_m(p)=0$ for all degrees of freedom in \eqref{d-5}.
Let \an{\label{p-5} p=p_5+\sum_{j=1}^{11} c_j b_{i_j} \quad \ \t{for some } \ p_5\in P_5(T).  }
Repeating \eqref{p-4-1} and \eqref{p-4-2},  we have 
\a{ p_5= \lambda_{14}^2 p_3 \quad \ \t{for some } \ p_3\in P_3(T). }
As $b_i$ have these four degrees of freedom vanished,  
  we then have
\a{ && p_3(\frac{2\b x_4+\b x_3}3)&=0, \ & p_2( \b x_3)&=0, && \\
     &&p_3(\frac{\b x_4+2\b x_3}3)&=0, \ & \partial_x  p_2( \b x_3)&=0, }
and consequently $p_3|_{\b x_4\b x_3}=0$.
We factor out this linear polynomial factor as
\a{ p_5= \lambda_{14}^2 \lambda_{43} p_2  \quad \ \t{for some } \ p_2\in P_2(T). }
Evaluating the normal derivative, we have
\a{ \partial_y p_5(\frac{2\b x_4+\b x_3}3)&= \frac 1{3^2} \cdot 
   \frac{-1} h p_2(\frac{2\b x_4+\b x_3}3)=0,\\
 \partial_y p_5(\frac{\b x_4+2\b x_3}3)&= \frac {2^2}{3^2} \cdot
       \frac{-1} h p_2(\frac{\b x_4+2\b x_3}3)=0,\\ 
    \partial_y p_5( \b x_3 )&= 1 \cdot \frac{-1} h p_2(\b x_3)=0,
  }where $h$ is the $y$-size of $T$.
We factor out this linear polynomial factor as
\a{ p_5= \lambda_{14}^2 \lambda_{43}^2 p_1 \quad \ \t{for some } \ p_1\in P_1(T). }
We evaluate the function values in the middle of edge $\b x_2\b x_3$, cf. Figure \ref{T5},
\a{ p_5(\frac{2\b x_2+\b x_3}3)&= 1^2 \cdot \frac{2^2}{3^2}\cdot p_1(\frac{2\b x_2+\b x_3}3)=0,\\
    p_5(\frac{\b x_2+2\b x_3}3)&= 1^2 \cdot \frac{1^2}{3^2}\cdot p_1(\frac{\b x_2+2\b x_3}3)=0. }
Thus $p_1$ vanishes on the edge and we have
\a{ p_5= \lambda_{14}^2 \lambda_{43}^2 \lambda_{23} p_0 \quad \ \t{for some } \ p_0\in P_0(T). }
Evaluating the last degree of freedom, cf. Figure \ref{T5},
\a{ \partial_x p_5(\frac{\b x_2 +2\b x_3}3)=1 \cdot \frac 1 {3^2} \cdot \frac {1} h p_0 =0, }
where $h$ is the size of square $T$.  Thus $p_0=0$ and $p_5=0$ in \eqref{p-5}.

Evaluating $p$ in \eqref{p-5} sequentially at the degrees of freedom of $b_{i_j}$,
  it follows that
  \a{ c_1=\dots=c_{11} = 0, \quad \t{and }\ p=0.  }
The lemma is proved.
\end{proof}

\section{The $C^1$-$P_k$ ($k\ge 6$) serendipity finite element}

For all $k\ge 6$, we enrich the $P_k$ space by 12 bubbles to define the $C^1$-$P_k$ ($k\ge 6$) 
  serendipity element,
\an{\label{V-k} V_k(T)=\t{span} \{ P_k(T), \ b_5, b_6, b_7, b_8, b_{12}, b_{14}, b_{2_1},
        b_{2_2},  b_{2_3}, b_{2_4}, b_{2_6}, b_{3_2}\},  }
where $b_{i}$ is a basis function in \eqref{W-h} dual to a vertex degree of freedom (first row in
   \eqref{d-Q-k}), and $b_{i_j}$ is a basis function in  
       \eqref{W-h}, dual to the $j$-th degree of freedom 
   $F_m(p)$ in the $i$-th row of \eqref{d-Q-k}, cf. Figure \ref{T-k}.
We define the following degrees of freedom for the space $V_k(T)$, which also ensure the global 
   $C^1$ continuity, \ cf. Figure \ref{T-k}, by $F_m(p)=$
\an{\label{d-k} &\begin{cases} p(\b x_i), \partial_x p(\b x_i), \partial_y p(\b x_i),
        \partial_{x y} p(\b x_i),   \ & i=1,2,3,4, \\
    p(\frac { j\b x_1+j'\b x_{2}}{k-2}),\;
    \partial_{y} p(\frac { j\b x_1+j'\b x_{2}}{k-2}), & j=1,\dots, k-3, \\ 
    p(\frac { j\b x_2+j'\b x_{3}}{k-2}), \;
    \partial_{x} p(\frac { j\b x_2+j'\b x_{3}}{k-2}), & j=1,\dots, k-3, \\  
     p(\frac { j\b x_4+j'\b x_{3}}{k-2}), \;
    \partial_{y} p(\frac { j\b x_4+j'\b x_{3}}{k-2}), & j=1,\dots, k-3, \\ 
    p(\frac { j\b x_1+j'\b x_{4}}{k-2}), \;
    \partial_{x} p(\frac { j\b x_1+j'\b x_{4}}{k-2}), & j=1,\dots, k-3,\\
     p(\frac {  i\b x_2+j\b x_{4}+(k-5-i-j)\b x_1 }{k-2}), 
            & i=1,\dots,k-7,\\ & j=1,\dots, i, \;k>7. 
    \end{cases} }
Notice that the $\dim P_{k-8}$ internal Lagrange points are located exactly at
  some of $C^1$-$Q_k$ interpolation points in \eqref{d-Q-k}.

\begin{lemma} The degrees of freedom \eqref{d-k} uniquely determine the 
   $V_k(T)$ functions in \eqref{V-k}.
\end{lemma}

\begin{proof} We count the dimension of $V_k$ in \eqref{V-k} and the number $N_{\text{dof}}$ of
   degrees of freedom in \eqref{d-k},
\a{ \dim V_k(T) &= \dim P_k + 12 =\frac{(k+1)(k+2)}2+12  \\
                &=\begin{cases} 40, & k=6, \\
                    48, & k=7, \\
                    \frac 12 k^2+\frac 32 k + 13, \qquad & k\ge 8, \end{cases}\\
    N_{\text{dof}} &=16+8(k-3)+\frac{(k-7)(k-6)}2 \\
                &=\begin{cases} 40, & k=6, \\
                    48, & k=7, \\
                    \frac 12 k^2+\frac 32 k + 13, \qquad & k\ge 8. \end{cases} }
Thus the uni-solvency is determined by uniqueness.

Let $p\in V_k(T)$ in \eqref{V-k} and $F_m(p)=0$ for all degrees of freedom in \eqref{d-k}.
Let \an{\label{p-k} p=p_k+\sum_{j=1}^{12} c_j b_{i_j} \quad \ \t{for some } \ p_k\in P_k(T).  }
Repeating \eqref{p-4-1} and \eqref{p-4-2},  we have 
\a{ p_k= \lambda_{14}^2 p_{k-2} \quad \ \t{for some } \ p_{k-2}\in P_{k-2}(T). }
As $b_{i_j}$ have these $(k-1)$ degrees of freedom vanished,   we   have
\a{ \partial_x  p_{k-2}( \b x_3)&=0, & p_{k-2}(\frac{j\b x_3+(k-2-j)\b x_4}{k-2})&=0, \;
   j=1,\dots,k-2,  }
and consequently $p_{k-2}|_{\b x_4\b x_3}=0$.
We factor out this linear polynomial factor as
\a{ p_{k}= \lambda_{14}^2 \lambda_{43} p_{k-3}  \quad \ \t{for some } \ p_{k-3}\in P_{k-3}(T). }

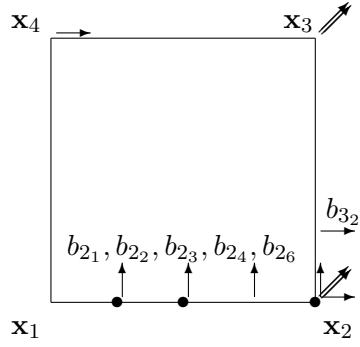
\begin{figure}[H]\centering
 \begin{picture}(140,140)(0,-10) 
 \def\c{\circle*{4}}\def\h{\vector(1,0){13}}\def\v{\vector(0,1){13}}
 \def\d{\multiput(0.5,-0.5)(-1,1){2}{\vector(1,1){11}}}

 \put(0,0){\begin{picture}(100,108)(0,0)   
  \put(0,0){\line(1,0){100}}  \put(0,0){\line(0,1){100}}   
  \put(100,100){\line(-1,0){100}}  \put(100,100){\line(0,-1){100}}  \put(102,2){\v} 
   \put(102,102){\d}  \put(102,2){\h} \put(102,2){\d} 
  \put(2,102){\h}  \put(100,0){\c}
  
    \put(27,2){\v}\put(52,2){\v}
    \put(25,0){\c} \put(50,0){\c}\put(77,2){\v} \put(6,18){$b_{2_1},
        b_{2_2},  b_{2_3}, b_{2_4}, b_{2_6} $}
  
      \put(102,27){\h}  \put(104,32){$  b_{3_2}$}
      \put(-15,-12){$\b x_1$} \put(103,-12){$\b x_2$}
      \put(-15,104){$\b x_4$} \put(88,104){$\b x_3$}
 \end{picture} }

 \end{picture}
 \caption{The 12 bubble functions $\{b_5, b_6, b_7, b_8, b_{2_1}$, $
        b_{2_2},  b_{2_3}, b_{2_4}, b_{2_6}, b_{3_2} \}$
      used to define $C^1$-$P_k$ ($k\ge 6$) serendipity element in \eqref{V-k}. } \label{T-k}
 \end{figure}

Evaluating the normal derivative, we have
\a{ &\quad \ \partial_y p_k(\frac{j\b x_3+(k-2-j)\b x_4}{k-2})\\
    &= \frac {j^2}{(k-2)^2} \cdot 
   \frac{-1} h p_{k-3}(\frac{j\b x_3+(k-2-j)\b x_4}{k-2})\\
   &=0,\qquad j=1,\dots,k-2,
  }and $p_{k-3}|_{\b x_4\b x_3}=0$.
We factor out this linear polynomial factor as
\a{ p_k= \lambda_{14}^2 \lambda_{43}^2 p_{k-4} \quad \ \t{for some } \ p_{k-4}\in P_{k-4}(T). }
We evaluate the function values in the internal points of edge $\b x_2\b x_3$, cf. Figure \ref{T-k},
\a{ &\quad \  p_k(\frac{j\b x_3+(k-2-j)\b x_2}{k-2})\\
    &= 1^2\cdot  \frac {j^2}{(k-2)^2} \cdot p_{k-4}(\frac{j\b x_3+(k-2-j)\b x_2}{k-2})\\
   &=0,\qquad\qquad j=1,\dots,k-3,
  }and $p_{k-4}|_{\b x_2\b x_3}=0$. Thus we have
\a{ p_k= \lambda_{14}^2 \lambda_{43}^2 \lambda_{23} p_{k-5} \quad \ \t{for some }
    \ p_{k-5}\in P_{k-5}(T). }
Evaluating the $x$-derivative degrees of freedom (one less, $b_{3_2}$), cf. Figure \ref{T-k}, we get
\a{  &\quad \  \partial_x p_k(\frac{j\b x_3+(k-2-j)\b x_2}{k-2})\\
    &= 1^2\cdot  \frac {j^2}{(k-2)^2}\cdot\frac 1 h 
          \cdot p_{k-5}(\frac{j\b x_3+(k-2-j)\b x_2}{k-2})\\
   &=0,\qquad\qquad j=2,\dots,k-3,
  }and $p_{k-5}|_{\b x_2\b x_3}=0$.
  
  Factoring out the factor again, we have
\a{ p_k= \lambda_{14}^2 \lambda_{43}^2 \lambda_{23}^2 p_{k-6}
    \quad \ \t{for some } \ p_{k-6}\in P_{k-6}(T). }
Evaluating the function-value degrees of freedom on edge $\b x_1\b x_4$
    (one more than the $y$-derivative degrees of derivative), cf. Figure \ref{T-k}, we get
\a{  &\quad \  p_k(\frac{j\b x_1+(k-2-j)\b x_2}{k-2})\\
    &= 1^2\cdot  \frac {j^2}{(k-2)^2}\cdot \frac{(k-2-j)^2}{(k-2)^2} 
          \cdot p_{k-6}(\frac{j\b x_3+(k-2-j)\b x_2}{k-2})\\
   &=0,\qquad\qquad j=3,\dots,k-3,
  }and $p_{k-6}|_{\b x_1\b x_2}=0$. Thus,
\a{ p_k= \lambda_{14}^2 \lambda_{43}^2 \lambda_{23}^2 \lambda_{12} p_{k-7}
    \quad \ \t{for some } \ p_{k-7}\in P_{k-7}(T). } 
Evaluating the $y$-derivative degrees of freedom on $\b x_1\b x_2$, cf. Figure \ref{T-k}, we get
\a{  &\quad \  \partial_y p_k(\frac{j\b x_1+(k-2-j)\b x_2}{k-2})\\
    &= \frac {j^2}{(k-2)^2} \cdot  \frac {(k-2-j)^2}{(k-2)^2}\cdot\frac 1 h\cdot
          \cdot p_{k-7}(\frac{j\b x_1+(k-2-j)\b x_2}{k-2})\\
   &=0,\qquad\qquad j=4,\dots,k-3,
  }and $p_{k-7}|_{\b x_1\b x_2}=0$. It leads to
\a{ p_k= \lambda_{14}^2 \lambda_{43}^2 \lambda_{23}^2 \lambda_{12}^2 p_{k-8}
    \quad \ \t{for some } \ p_{k-8}\in P_{k-8}(T). } 
As the four factors are positive at the $\dim P_{k-8}$ internal Lagrange nodes in
  the last line of degrees of freedom \eqref{d-k}, and $b_{i_j}$ in \eqref{p-k} vanish
    at these $\dim P_{k-8}$ points in \eqref{d-Q-k}, we have $p_{k-8}=0$ at these points and
    $p_{k-8}=0$. 
 Thus, $p_k=0$ in \eqref{p-k}.

Evaluating $p$ in \eqref{p-k} sequentially at the degrees of freedom of $b_{i_j}$,
  it follows that
  \a{ c_1=\dots=c_{12} = 0, \quad \t{and }\ p=0.  } 
The proof is complete.
\end{proof}

\section{The finite element solution and convergence }

The $C^1$-$P_k$ serendipity finite element space is defined by, for all $k\ge 4$, 
\an{\label{V-h} V_h=\{v_h\in H^2_0(\Omega) : v_h|_T \in V_k(T) \quad \forall T\in \mathcal Q_h \},
   } where $V_k(T)$ is defined in \eqref{V-4}, or \eqref{V-5}, or \eqref{V-k}.

 The finite element discretization of the biharmonic equation \eqref{bi} reads:
   Find $u\in V_h$ such that
\an{\label{finite} (\Delta u, \Delta v) = (f, v) \quad \forall v\in V_h, }
where $V_h$ is defined in \eqref{V-h}.
            
\begin{lemma}  The finite element problem \eqref{finite} has a unique solution.
\end{lemma}
                        
\begin{proof}  As \eqref{finite} is a square system of finite linear equations,  we only need to prove the uniqueness.
Let $f=0$ and $v_h=u_h$ in \eqref{finite}.
It follows $\Delta u_h = 0 $ on the domain.  Let $v\in H^2_0(\Omega)$ be the solution
  of \eqref{bi} with $f =\Delta u_h$, as $u_h\in H^2_0(\Omega)$.
Because $u_h\in C^1(\Omega)$,  we have
\a{ 0 &= \int_{\Omega} \Delta u_h v d\b x 
     = \int_{\Omega} -\nabla u_h \nabla v d\b x  
      =\int_{\Omega} (u_h)^2 d\b x.  }
Thus, $u_h=0$. The proof is complete.
\end{proof}

For convergence, the analysis is standard,  as we have $C^1$ conforming finite elements.

\begin{theorem}  Let $ u\in H^{k+1}\cap H^2_0(\Omega)$ be
    the exact solution of the biharmoic equation \eqref{bi}.  
   Let $u_h$ be the $C^1$-$P_k$ finite element solution of \eqref{finite}.   
   Assuming the full-regularity on \eqref{bi}, it holds 
  \a{  \| u- u_h\|_{0} + h^2  |  u- u_h |_{2}  
         & \le Ch^{k+1} | u|_{k+1}, \quad k\ge 6.  } 
\end{theorem}
                        
\begin{proof} As $V_h\subset H^2_0(\Omega)$,  from \eqref{bi} and \eqref{finite},  we get
\a{ (\Delta ( u- u_h), \Delta v_h)=0\quad \forall v_h\in   V_{h }. }
Applying the Schwartz inequality,  we get 
\a{    |   u- u_h|_{2}^2  
     & = C (\Delta(  u-  u_h), \Delta(  u- u_h ))\\ &= C (\Delta(  u-  u_h), \Delta(  u- I_h u))\\ 
     &\le C |   u- u_h|_{2} |   u- I_h u |_{2} \\ 
     & \le Ch^{k-1} |u|_{k+1} |   u- u_h|_{2}  ,} 
      where $ I_h  u$ is the nodal interpolation defined by DOFs in \eqref{d-4} or  \eqref{d-5} or
         \eqref{d-k}.
As $V_k(T)\supset P_k(T)$,  we have $I_h u|_T = u|_T$ if $u\in P_k(T)$, 
  i.e., $I_h$ preserves $P_k$ functions locally.
  Such an interpolation operator is
   $H^2$ stable and consequently of the
    optimal order of convergence, by modifying the standard theory in \cite{Girault,Scott-Zhang}.

For the $L^2$ convergence,  we need an $H^4$ regularity for the dual problem: Find 
  $w\in H^2_0(\Omega)$ such that
  \an{ \label{d2}
    (\Delta w, \Delta v) &=(u-u_h, v), \ \forall v \in H^2_0(\Omega), }
  where \a{ |w|_4 \le C \|u-u_h\|_0 . }
Thus, by \eqref{d2}, \a{ \|u-u_h\|_0^2 &=(\Delta w, \Delta (u-u_h) ) = 
(\Delta (w-w_h), \Delta (u-u_h) ) \\
  & \le C h^2 |w|_{4}  h^{k-1} | u | _{k+1} 
  \\&  \le C h^{ k+1 }   | u | _{k+1} \|u-u_h\|_0. }
                 The proof is complete.
\end{proof}

\section{Numerical Experiments}

In the numerical computation,  we solve the biharmonic equation \eqref{bi}
   on the unit square domain $\Omega=(0,1)\times(0,1)$. 
We choose an $f$ in \eqref{bi} so that the exact solution is
\an{\label{s2}
   u =\sin^2(\pi x)\sin^2(\pi y).  }  
   
\begin{figure}[H]
\begin{center}\setlength\unitlength{1.0pt}\centering
\begin{picture}(330,115)(0,0) \put(5,101){$G_1:$}  \put(115,101){$G_2:$} \put(225,101){$G_3:$} 
\put(0,-2){ \begin{picture}(100,100)(0,0) \multiput(0,0)(100,0){2}{\line(0,1){100}}
       \multiput(0,0)(0,100){2}{\line(1,0){100}} \end{picture} }
\put(110,-2){ \begin{picture}(100,100)(0,0) \multiput(0,0)(50,0){3}{\line(0,1){100}}
       \multiput(0,0)(0,50){3}{\line(1,0){100}} \end{picture} }
\put(220,-2){ \begin{picture}(100,100)(0,0) \multiput(0,0)(25,0){5}{\line(0,1){100}}
       \multiput(0,0)(0,25){5}{\line(1,0){100}} \end{picture} }
\end{picture}\end{center}
\caption{The first three square grids for computing  \eqref{s2} in Tables \ref{t1}--\ref{t5}. }
\label{f-21}
\end{figure}
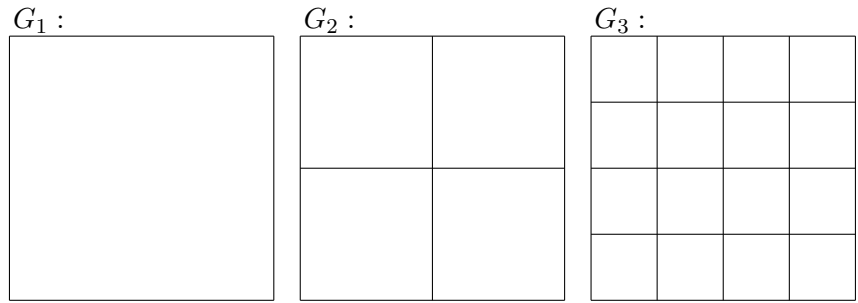

 We compute the solution \eqref{s2} on the square grids shown in Figure \ref{f-21}, by 
  the newly constructed $C^1$-$P_k$, $k=4,5,6,7,8$, serendipity finite elements \eqref{V-h}.
The results are listed in Tables \ref{t1}--\ref{t5}, where we can see that the optimal orders of convergence 
  are achieved in all cases.  
Additionally, we computed the corresponding $C^1$-$Q_k$ BFS finite element solutions in
  these tables.
The two solutions are about equally good. But the $P_4$ serendipity finite element saves about
   $1/10$ of unknowns comparing to the $Q_4$ element in Table \ref{t1}.
When $k$ is large, the global space of the $P_k$ serendipity finite element is
 about 1/2 of the size of that of the $Q_k$ BFS finite element.
In the last row of some tables, the computer accuracy is reached, i.e., the round-off error is
  more than the  truncation error.

\begin{table}[H]
  \centering  \renewcommand{\arraystretch}{1.1}
  \caption{Error profile on the square meshes shown as in Figure \ref{f-21}, 
     for computing \eqref{s2}. }
  \label{t1}
\begin{tabular}{c|cc|cc|r}
\hline
grid & \multicolumn{2}{c|}{ $\| u-u_h\|_{0}$  \; $O(h^r)$}  
  &  \multicolumn{2}{c}{  $|u-u_h|_{2}$ \;$O(h^r)$} & $\dim V_h$  \\ \hline
    &  \multicolumn{5}{c}{ By the $C^1$-$Q_4$ BFS element \eqref{W-h}. }   \\
\hline   
 1&    0.837E-01 &  0.0&    0.287E+01 &  0.0 &       25\\
 2&    0.939E-02 &  3.2&    0.161E+01 &  0.8 &       64\\
 3&    0.150E-03 &  6.0&    0.147E+00 &  3.5 &      196\\
 4&    0.461E-05 &  5.0&    0.184E-01 &  3.0 &      676\\
 5&    0.143E-06 &  5.0&    0.231E-02 &  3.0 &     2500\\
 6&    0.447E-08 &  5.0&    0.288E-03 &  3.0 &     9604\\
 7&    0.162E-09 &  4.8&    0.360E-04 &  3.0 &    37636\\
\hline 
    &  \multicolumn{5}{c}{ By the $C^1$-$P_4$ serendipity element \eqref{V-h}. }   \\
\hline   
 1&    0.375E+00 &  0.0&    0.174E+02 &  0.0 &       24\\
 2&    0.468E-01 &  3.0&    0.470E+01 &  1.9 &       60\\
 3&    0.704E-03 &  6.1&    0.239E+00 &  4.3 &      180\\
 4&    0.111E-04 &  6.0&    0.212E-01 &  3.5 &      612\\
 5&    0.290E-06 &  5.3&    0.248E-02 &  3.1 &     2244\\
 6&    0.869E-08 &  5.1&    0.306E-03 &  3.0 &     8580\\
 7&    0.382E-09 &  4.5&    0.381E-04 &  3.0 &    33540\\
\hline 
    \end{tabular}%
\end{table}%

\begin{table}[H]
  \centering  \renewcommand{\arraystretch}{1.1}
  \caption{Error profile on the square meshes shown as in Figure \ref{f-21}, 
     for computing \eqref{s2}. }
  \label{t2}
\begin{tabular}{c|cc|cc|r}
\hline
grid & \multicolumn{2}{c|}{ $\| u-u_h\|_{0}$  \; $O(h^r)$}  
  &  \multicolumn{2}{c}{  $|u-u_h|_{2}$ \;$O(h^r)$} & $\dim V_h$  \\
   \hline
    &  \multicolumn{5}{c}{ By the $C^1$-$Q_5$ BFS element \eqref{W-h}. }   \\
\hline   
 1&    0.324E-01 &  0.0&    0.435E+01 &  0.0 &       36\\
 2&    0.138E-03 &  7.9&    0.918E-01 &  5.6 &      100\\
 3&    0.789E-05 &  4.1&    0.146E-01 &  2.7 &      324\\
 4&    0.130E-06 &  5.9&    0.912E-03 &  4.0 &     1156\\
 5&    0.206E-08 &  6.0&    0.570E-04 &  4.0 &     4356\\
 6&    0.302E-10 &  6.1&    0.356E-05 &  4.0 &    16900\\
\hline 
    &  \multicolumn{5}{c}{ By the $C^1$-$P_5$ serendipity element \eqref{V-h}. }   \\
\hline   
 1&    0.375E+00 &  0.0&    0.136E+02 &  0.0 &       32\\
 2&    0.433E-01 &  3.1&    0.459E+01 &  1.6 &       84\\
 3&    0.419E-03 &  6.7&    0.227E+00 &  4.3 &      260\\
 4&    0.492E-05 &  6.4&    0.106E-01 &  4.4 &      900\\
 5&    0.697E-07 &  6.1&    0.562E-03 &  4.2 &     3332\\
 6&    0.103E-08 &  6.1&    0.323E-04 &  4.1 &    12804\\
\hline 
    \end{tabular}%
\end{table}%

\begin{table}[H]
  \centering  \renewcommand{\arraystretch}{1.1}
  \caption{Error profile on the square meshes shown as in Figure \ref{f-21}, 
     for computing \eqref{s2}. }
  \label{t3}
\begin{tabular}{c|cc|cc|r}
\hline
grid & \multicolumn{2}{c|}{ $\| u-u_h\|_{0}$  \; $O(h^r)$}  
  &  \multicolumn{2}{c}{  $|u-u_h|_{2}$ \;$O(h^r)$} & $\dim V_h$  \\
   \hline
    &  \multicolumn{5}{c}{ By the $C^1$-$Q_6$ BFS element \eqref{W-h}. }   \\
\hline   
 1&    0.157E-02 &  0.0&    0.802E+00 &  0.0 &       49\\
 2&    0.706E-04 &  4.5&    0.499E-01 &  4.0 &      144\\
 3&    0.394E-06 &  7.5&    0.115E-02 &  5.4 &      484\\
 4&    0.310E-08 &  7.0&    0.360E-04 &  5.0 &     1764\\
 5&    0.258E-10 &  6.9&    0.113E-05 &  5.0 &     6724\\
\hline 
    &  \multicolumn{5}{c}{ By the $C^1$-$P_6$ serendipity element \eqref{V-h}. }   \\
\hline   
 1&    0.375E+00 &  0.0&    0.137E+02 &  0.0 &       40\\
 2&    0.131E-01 &  4.8&    0.158E+01 &  3.1 &      108\\
 3&    0.222E-03 &  5.9&    0.370E-01 &  5.4 &      340\\
 4&    0.208E-05 &  6.7&    0.800E-03 &  5.5 &     1188\\
 5&    0.169E-07 &  6.9&    0.197E-04 &  5.3 &     4420\\
\hline 
    \end{tabular}%
\end{table}%

\begin{table}[H]
  \centering  \renewcommand{\arraystretch}{1.1}
  \caption{Error profile on the square meshes shown as in Figure \ref{f-21}, 
     for computing \eqref{s2}. }
  \label{t4}
\begin{tabular}{c|cc|cc|r}
\hline
grid & \multicolumn{2}{c|}{ $\| u-u_h\|_{0}$  \; $O(h^r)$}  
  &  \multicolumn{2}{c}{  $|u-u_h|_{2}$ \;$O(h^r)$} & $\dim V_h$  \\
   \hline
    &  \multicolumn{5}{c}{ By the $C^1$-$Q_7$ BFS element \eqref{W-h}. }   \\
\hline   
 1&    0.115E-02 &  0.0&    0.379E+00 &  0.0 &       64\\
 2&    0.964E-06 & 10.2&    0.253E-02 &  7.2 &      196\\
 3&    0.183E-07 &  5.7&    0.763E-04 &  5.0 &      676\\
 4&    0.731E-10 &  8.0&    0.119E-05 &  6.0 &     2500\\ 
 5&    0.158E-10 &  2.2&    0.185E-07 &  6.0 &     9604\\
\hline 
    &  \multicolumn{5}{c}{ By the $C^1$-$P_7$ serendipity element \eqref{V-h}. }   \\
\hline   
 1&    0.375E+00 &  0.0&    0.140E+02 &  0.0 &       48\\
 2&    0.380E-02 &  6.6&    0.430E+00 &  5.0 &      132\\
 3&    0.668E-05 &  9.2&    0.426E-02 &  6.7 &      420\\
 4&    0.247E-07 &  8.1&    0.541E-04 &  6.3 &     1476\\ 
 5&    0.313E-10 &  9.6&    0.735E-06 &  6.2 &     5508\\
\hline 
    \end{tabular}%
\end{table}%

\begin{table}[H]
  \centering  \renewcommand{\arraystretch}{1.1}
  \caption{Error profile on the square meshes shown as in Figure \ref{f-21}, 
     for computing \eqref{s2}. }
  \label{t5}
\begin{tabular}{c|cc|cc|r}
\hline
grid & \multicolumn{2}{c|}{ $\| u-u_h\|_{0}$  \; $O(h^r)$}  
  &  \multicolumn{2}{c}{  $|u-u_h|_{2}$ \;$O(h^r)$} & $\dim V_h$  \\
   \hline
    &  \multicolumn{5}{c}{ By the $C^1$-$Q_8$ BFS element \eqref{W-h}. }   \\
\hline   
 1&    0.531E-04 &  0.0&    0.716E-01 &  0.0 &       81\\
 2&    0.546E-06 &  6.6&    0.743E-03 &  6.6 &      256\\
 3&    0.755E-09 &  9.5&    0.433E-05 &  7.4 &      900\\
 4&    0.557E-11 &  7.1&    0.334E-07 &  7.0 &     3364\\
\hline 
    &  \multicolumn{5}{c}{ By the $C^1$-$P_8$ serendipity element \eqref{V-h}. }   \\
\hline   
 1&    0.465E-01 &  0.0&    0.389E+01 &  0.0 &       57\\
 2&    0.365E-03 &  7.0&    0.465E-01 &  6.4 &      160\\
 3&    0.133E-05 &  8.1&    0.421E-03 &  6.8 &      516\\
 4&    0.229E-08 &  9.2&    0.292E-05 &  7.2 &     1828\\
\hline 
    \end{tabular}%
\end{table}%

\end{document}